\documentclass[12pt,leqno]{article}
\usepackage{amsfonts}
\usepackage{amsmath, amsthm, amsfonts, amssymb, color}
\usepackage{mathrsfs}
\newtheorem{thm}{Theorem}[section]

\newtheorem{lem}[thm]{Lemma}

\theoremstyle{definition}

\newcommand{\scr}[1]{\mathscr #1}
\definecolor{wco}{rgb}{0.5,0.2,0.3}

\numberwithin{equation}{section} \theoremstyle{remark}
\newtheorem{rem}{Remark}[section]

\newcommand{\ua}{\uparrow}

\title{{\bf Convergence Rate of EM Scheme for  SDDEs}
}
\author{
{\bf  Jianhai Bao and Chenggui Yuan}\\
 \footnotesize{Department of Mathematics,
Swansea University, Singleton Park, SA2 8PP, UK}\\
\footnotesize{ majb@Swansea.ac.uk, C.Yuan@Swansea.ac.uk}}
\begin{document}
\def\R{\mathbb R}  \def\ff{\frac} \def\ss{\sqrt} \def\B{\mathbf
B}
\def\N{\mathbb N} \def\kk{\kappa} \def\m{{\bf m}}
\def\dd{\delta} \def\DD{\Delta} \def\vv{\varepsilon} \def\rr{\rho}
\def\<{\langle} \def\>{\rangle} \def\GG{\Gamma} \def\gg{\gamma}
  \def\nn{\nabla} \def\pp{\partial} \def\EE{\scr E}
\def\d{\text{\rm{d}}} \def\bb{\beta} \def\aa{\alpha} \def\D{\scr D}
  \def\si{\sigma} \def\ess{\text{\rm{ess}}}
\def\beg{\begin} \def\beq{\begin{equation}}  \def\F{\scr F}
\def\Ric{\text{\rm{Ric}}} \def\Hess{\text{\rm{Hess}}}
\def\e{\text{\rm{e}}} \def\ua{\underline a} \def\OO{\Omega}  \def\oo{\omega}
 \def\tt{\tilde} \def\Ric{\text{\rm{Ric}}}
\def\cut{\text{\rm{cut}}} \def\P{\mathbb P} \def\ifn{I_n(f^{\bigotimes n})}
\def\C{\scr C}      \def\aaa{\mathbf{r}}     \def\r{r}
\def\gap{\text{\rm{gap}}} \def\prr{\pi_{{\bf m},\varrho}}  \def\r{\mathbf r}
\def\Z{\mathbb Z} \def\vrr{\varrho} \def\ll{\lambda}
\def\L{\scr L}\def\Tt{\tt} \def\TT{\tt}\def\II{\mathbb I}
\def\i{{\rm in}}\def\Sect{{\rm Sect}}\def\E{\mathbb E} \def\H{\mathbb H}
\def\M{\scr M}\def\Q{\mathbb Q} \def\texto{\text{o}} \def\LL{\Lambda}
\def\Rank{{\rm Rank}} \def\B{\scr B} \def\i{{\rm i}} \def\HR{\hat{\R}^d}
\def\to{\rightarrow}\def\l{\ell}
\def\8{\infty}

\maketitle

\begin{abstract} In this paper we investigate the convergence rate
of Euler-Maruyama scheme for a class of stochastic differential {\it
delay} equations, where the corresponding coefficients may be {\it
highly nonlinear} with respect to the delay variables. In
particular, we reveal that the convergence rate of Euler-Maruyama
scheme is $\frac{1}{2}$ for the Brownian motion case, while show
that
 it is best to use the mean-square convergence for the pure jump case,
and that the order of mean-square convergence is close to
$\frac{1}{2}$.
\end{abstract}
 \vskip 2cm

\section{Introduction}
Since most stochastic differential equations (SDEs) can not be
solved explicitly, numerical methods have become essential.
Recently, there is extensive literature in investigating the strong
convergence, weak convergence or sample path convergence  of
numerical schemes for SDEs, e.g., in \cite{g98} for SDEs with a
monotone condition, in \cite{hk05,jwy09,PB10} for SDEs with jumps,
in \cite{hu96, jwy09,KP00,ms03} for stochastic differential delay
equations (SDDEs) 
and in \cite{hk05,hms02} for SDEs with a one-side Lipschitz
condition. For the comprehensive monographs on numerical approximate
methods of SDEs, we can also refer to \cite{KP99,PB10,S97}. Although
the results on convergence of Euler-Maruyama (EM) schemes are
substantial, there are limited ones on convergence rate  under
weaker conditions than global Lipschitz condition and linear growth
condition. For example, a recent work in \cite{gr11} 
reveals the convergence rate of EM schemes for a class
of SDEs under a H\"older condition, and, with local Lipschitz
constants satisfying  a logarithm growth condition, \cite{ym08} and
\cite{bbmy11,jwy09} discuss the convergence rate of EM approximate
methods for SDEs and stochastic functional differential equations
with jumps, respectively. We should also point out that the strong
convergence of EM schemes for SDDEs is, in general, discussed under
a linear growth condition or bounded moments of analytic and
numerical solutions, e.g., \cite{jwy09,KP00,ms03}, and that the
convergence rate \cite{bbmy11,jwy09} is also revealed under a linear
growth condition.

To further motivate our work, we first consider an SDDE on
$\mathbb{R}$
\begin{equation}\label{eq35}
\d X(t)=\{aX(t)+bX^3(t-\tau)\}\d t+cX^2(t-\tau)\d W(t),
\end{equation}
where $a,b,c\in\mathbb{R}, \tau>0$, are constant, and $W(t)$ is a
scalar Brownian motion. It is easy to observe that both the drfit
coefficient and the diffusion coefficient are {\it highly nonlinear}
especially with respect to the delay arguments. Therefore, the
existing convergence results, e.g., \cite{jwy09,KP00,ms03}, can not
cover Eq. \eqref{eq35}, and the convergence rate of the
corresponding EM scheme can not also be revealed by the techniques
of \cite{bbmy11,jwy09} as we have explained in the end of the last
paragraph. On the other hand, our work is also enlightened by the
recent work in \cite{gr11} such that consider SDE on $\mathbb{R}$
\begin{equation*}
\d X(t)=\{f(t,X(t))+g(t,X(t))\}\d t+\sigma(t,X(t))\d W(t),
\end{equation*}
and discuss the convergence rate of the associated EM method, where
$g$ is H\"older continuous, of {\it linear growth, and monotone
decreasing with respect to the second variable}.

Motivated by the previous literature, in this paper we not only
study the strong convergence of EM schemes for a class of SDDEs,
which may be {\it highly nonlinear} with respect to the delay
variables, but also reveal the {\it convergence rate} of the
corresponding EM numerical methods. The rest of the
paper are organized as follows: under highly nonlinear growth
conditions with respect to the delay arguments, in Section 2 we
reveal the convergence rate of EM schemes for SDDEs driven by
Brownian motion is $\frac{1}{2}$, while in Section 3 we show that it
is best to use the mean-square convergence for the pure jump case,
and that the rate of mean-square convergence is close to
$\frac{1}{2}$.

\section{Convergence Rate for Brownian Motion Case}
For integer $n>0$, let
$(\mathbb{R}^n,\langle\cdot,\cdot\rangle,|\cdot|)$ be the Euclidean
space and $\|A\|:=\sqrt{\mbox{trace}(A^*A)}$ the Hilbert-Schmidt
norm for a matrix $A$, where $A^*$ is its transpose. Let $W(t)$ be
an $m$-dimensional  Brownian motion defined on some complete
probability space $(\Omega,\mathcal {F},\mathbb{P},\{\mathcal
{F}_t\}_{t\geq0})$. Throughout the paper, $C>0$ denotes a generic
constant whose values may change from lines to lines.

For fixed $T>0$, in this section we consider SDDE on $\mathbb{R}^n$
\begin{equation}\label{eq1}
\begin{split}
\d X(t)&=b(X(t),X(t-\tau))\d t+\sigma (X(t),X(t-\tau))\d W(t), \ \
t\in[0,T]
\end{split}
\end{equation}
with initial data $X(\theta)=\xi(\theta),\theta\in[-\tau,0]$.

To guarantee the existence and uniqueness of solution we introduce
the following conditions. Let
$V_i:\mathbb{R}^n\times\mathbb{R}^n\rightarrow\mathbb{R}_+$ such
that
\begin{equation}\label{eq19}
V_i(x,y)\leq K_i(1+|x|^{q_i}+|y|^{q_i}),\ \ \ i=1,2
\end{equation}
for some $K_i>0, q_i\geq1$ and arbitrary $x,y\in\mathbb{R}^n$. We
further assume that
\begin{enumerate}
\item[\textmd{(A1)}] $b:\mathbb{R}^n\times\mathbb{R}^n\rightarrow
\mathbb{R}^n$ and there exists $L_1>0$ such that
\begin{equation*}
|b(x_1,y_1)-b(x_2,y_2)|\leq L_1|x_1-x_2|+V_1(y_1,y_2)|y_1-y_2|
\end{equation*}
for $x_i,y_i\in \mathbb{R}^n,i=1,2$;
\item[\textmd{(A2)}] $\sigma:\mathbb{R}^n\times\mathbb{R}^n\rightarrow
\mathbb{R}^{n\times m}$ and there exists $L_2>0$ such that
\begin{equation*}
\|\sigma(x_1,y_1)-\sigma(x_2,y_2)\|\leq
L_2|x_1-x_2|+V_2(y_1,y_2)|y_1-y_2|
\end{equation*}
for $x_i,y_i\in \mathbb{R}^n,i=1,2$,
\end{enumerate}

We now introduce an EM  method for Eq. \eqref{eq1}. Without loss of
generality, we may assume that there exist sufficiently large
integers $N,M>0$ such that
\begin{equation}\label{eq27}
\triangle:=\frac{\tau}{N}=\frac{T}{M}\in(0,1).
\end{equation}
Define a continuous EM scheme associated with Eq. \eqref{eq1}
\begin{equation}\label{eq20}
\begin{split}
\d Y(t)&=b(\bar{Y}(t),\bar{Y}(t-\tau))\d t+\sigma
(\bar{Y}(t),\bar{Y}(t-\tau))\d W(t),
\end{split}
\end{equation}
where $\bar{Y}(t):=Y(k\triangle)$ for $
t\in[k\triangle,(k+1)\triangle), k=0,1,\cdots,M-1$, and
$\bar{Y}(\theta)=\xi(\theta),\theta\in[-\tau,0]$.

\begin{rem}
Clearly, if $b$ and $\sigma$ are globally Lipschitzian,  then $b$ and $\sigma$ are satisfied with  $(A1)$ and $(A2)$. On the other hand, we
remark that $b$ and $\sigma$ may be {\it highly nonlinear} with
respect to the delay variables. There are many such examples which
are covered by $(A1)-(A2)$. For example, for Eq. \eqref{eq35}  it is
trivial to see that $b(x,y)=ax+by^3$, $\sigma(x,y)=cy^2$, and
$(A1)-(A2)$ hold by choosing $V_1(x,y)=\frac{3|b|}{2}(x^2+y^2)$ and
$ V_2(x,y)=|c|(|x|+|y|)$. In fact, the examples, where the drift
coefficient and the diffusion coefficient are polynomial of degree
$d\geq1$ with regard to the delay variables, are  included in
our framework.
\end{rem}

\begin{lem}\label{lemma1.1}
{\rm Assume that $(A1)$ and $(A2)$ hold. Then, for any initial data
$\xi\in C^b_{\mathcal {F}_0}([-\tau,0];\mathbb{R}^n)$, Eq.
\eqref{eq1} admits a unique global strong solution $X(t),
t\in[0,T]$. Moreover, for any $p\geq2$ there exists $C>0$ such that
\begin{equation}\label{eq3}
\mathbb{E}\Big(\sup_{0\leq t\leq
T}|X(t)|^p\Big)\vee\mathbb{E}\Big(\sup_{0\leq t\leq
T}|Y(t)|^p\Big)\leq C,
\end{equation}
and
\begin{equation}\label{eq9}
\mathbb{E}|Y(t)-\bar{Y}(t)|^p\leq C\triangle^{\frac{p}{2}}.
\end{equation}
}
\end{lem}
\begin{proof}
Note that Eq. \eqref{eq1} has a unique local solution due to the
fact that both $b$ and $\sigma$ are locally Lipschitzian. To verify
that Eq. \eqref{eq1} admits a unique global solution on time
interval $[0,T]$, it is sufficient to show that
\begin{equation}\label{eq14}
\mathbb{E}\Big(\sup_{0\leq t\leq T}|X(t)|^p\Big)\leq C,\ \ \ p\geq2.
\end{equation}
By a straightforward computation,  we can deduce from $(A1),(A2)$
and \eqref{eq19} that
\begin{equation}\label{eq16}
|b(x,y)|\leq C(1+|x|+|y|+|y|^{q_1+1}), \ \ \ x,y\in\mathbb{R}^n,
\end{equation}
and
\begin{equation}\label{eq17}
\|\sigma(x,y)\|\leq C(1+|x|+|y|+|y|^{q_2+1}), \ \ \
x,y\in\mathbb{R}^n.
\end{equation}
Set $\gamma_1:=q_1+1$ and $\gamma_2:=q_2+1$. To show \eqref{eq14},
by \eqref{eq16} and \eqref{eq17}, the H\"older inequality and the
Burkhold-Davis-Gundy inequality, we have that for any $p\geq2$ and
$t\in[0,T]$
\begin{equation*}
\begin{split}
\mathbb{E}\Big(\sup_{0\leq s\leq t}|X(s)|^p\Big)&\leq
3^{p-1}\Big\{|\xi(0)|^p+\mathbb{E}\Big(\sup_{0\leq s\leq t}\Big|\int_0^sb(X(r),X(r-\tau))\d r\Big|^p\Big)\\
&\quad+\mathbb{E}\Big(\sup_{0\leq s\leq t}\Big|\int_0^s\sigma(X(r),X(r-\tau))\d W(r)\Big|^p\Big)\Big\}\\
&\leq
C\Big\{1+\mathbb{E}\int_0^t(|b(X(s),X(s-\tau))|^p+\|\sigma(X(s),X(s-\tau))\|^p)\d
s\Big\}\\
&\leq C\Big\{1+\mathbb{E}\int_0^t|X(s)|^p\d
s+\mathbb{E}\int_0^t(|X(s-\tau)|^{p\gamma_1}+|X(s-\tau)|^{p\gamma_2})\d
s\Big\},
\end{split}
\end{equation*}
where we have also used the Young inequality in the last step.
 This, together with the
Gronwall inequality, yields that for $t\in[0,T]$ and $p\geq2$
\begin{equation}\label{eq13}
\mathbb{E}\Big(\sup_{0\leq s\leq t}|X(s)|^p\Big)\leq
C\Big\{1+\mathbb{E}\int_0^t(|X(s-\tau)|^{p\gamma_1}+|X(s-\tau)|^{p\gamma_2})\d
s\Big\}.
\end{equation}
The following argument is similar to that of \cite[Theorem
2.1]{wmc09}, however we give a detailed proof, which will also be used in the   proof of Theorem \ref{Th2} below. Let
$\beta:=\gamma_1\vee\gamma_2$, and
\begin{equation*}
p_i:=([T/\tau]+2-i)p\beta^{[T/\tau]+1-i}, \ \ \
i=1,2,\cdots,[T/\tau]+1,
\end{equation*}
where $[a]$ denotes the integer part of real number $a$. Thus, due
to $\beta\geq1$ and $p\geq2$, it is easy to see that $p_i\geq2$ such
that
\begin{equation*}
p_{i+1}\beta<p_i \mbox{ and } p_{[T/\tau]+1}=p, \ \ \
i=1,2,\cdots,[T/\tau].
\end{equation*}
By \eqref{eq13}, together with $\xi\in C^b_{\mathcal
{F}_0}([-\tau,0];\mathbb{R}^n)$,  we obtain that
\begin{equation*}
\mathbb{E}\Big(\sup_{0\leq s\leq \tau}|X(s)|^{p_1}\Big)\leq C,
\end{equation*}
which, combining \eqref{eq13} with the H\"older inequality, further
leads to
\begin{equation*}
\begin{split}
\mathbb{E}\Big(\sup_{0\leq s\leq 2\tau}|X(s)|^{p_2}\Big)&\leq
C\Big\{1+\mathbb{E}\int_0^{2\tau}(|X(s-\tau)|^{p_2\gamma_1}+|X(s-\tau)|^{p_2\gamma_2})\d
s\Big\}\\
&\leq
C\Big\{1+\int_0^{\tau}\Big((\mathbb{E}|X(s)|^{p_1})^{\frac{p_2\gamma_1}{p_1}}+(\mathbb{E}|X(s)|^{p_1})^{\frac{p_2\gamma_2}{p_1}}\Big)\d
s\Big\}\\
&\leq C.
\end{split}
\end{equation*}
Repeating the previous procedures gives \eqref{eq14} and
$\mathbb{E}\Big(\sup_{0\leq t\leq T}|Y(t)|^p\Big)\leq C$. Finally,
the statement \eqref{eq9} can also be obtained by taking into
account the H\"older inequality, the Burkhold-Davis-Gundy inequality
and  \eqref{eq3}.
\end{proof}

\begin{rem}
{\rm Lemma \ref{lemma1.1} gives a new result on existence and
uniqueness of solutions to SDDEs on finite-time interval, where the
coefficients may be polynomial of any degree $d\geq1$ with regard to
the delay variables. }
\end{rem}

We can now state our main result, which not only shows the strong
convergence of EM scheme associated with Eq. \eqref{eq1} but also
reveals its convergence rate,  although the drift coefficient and
the diffusion coefficient may be highly nonlinear with respect to
the delay arguments.
\begin{thm}\label{Th2}
{\rm Under $(A1)$ and $(A2)$, for any $p\geq2$ there exits $C>0$
such that
\begin{equation*}
\mathbb{E}\Big(\sup\limits_{0\leq s\leq T}|X(s)-Y(s)|^p\Big)\leq
C\triangle^{\frac{p}{2}},
\end{equation*}
that is, the rate of convergence  of EM scheme \eqref{eq20} is
$\frac{1}{2}$. }
\end{thm}

\begin{proof}
The argument is motivated by that of \cite[Theorem 2.1]{gr11}. For
fixed $\delta>1$ and arbitrary $\epsilon\in(0,1)$, there exists a
continuous nonnegative function $\psi_{\delta\epsilon}(x),x\geq0$,
with support $[\epsilon/\delta,\epsilon]$, such that
\begin{equation*}
\int_{\epsilon/\delta}^{\epsilon}\psi_{\delta\epsilon}(x)\d x=1
\mbox{ and }\psi_{\delta\epsilon}(x)\leq\frac{2}{x\ln \delta},\ \
x>0.
\end{equation*}
Define
\begin{equation*}
\phi_{\delta\epsilon}(x):=\int_0^{x}\int_0^y\psi_{\delta\epsilon}(z)\d
z\d y,\ \ \ \ x>0.
\end{equation*}
Then $\phi_{\delta\epsilon}\in C^2(\mathbb{R}_+;\mathbb{R}_+)$
 possesses the following properties:
\begin{equation}\label{eq4}
x-\epsilon\leq\phi_{\delta\epsilon}(x)\leq x, \ \ \ \ \ \ x>0,
\end{equation}
and
\begin{equation}\label{eq5}
0\leq\phi'_{\delta\epsilon}(x)\leq1, \ \ \
\phi''_{\delta\epsilon}(x)\leq\frac{2}{x\ln\delta}{\bf1}_{[\epsilon/\delta,\epsilon]}(x),
\ \ \ x>0.
\end{equation}
Define
\begin{equation}\label{eq21}
V_{\delta\epsilon}(x):=\phi_{\delta\epsilon}(|x|), \ \ \
x\in\mathbb{R}^n.
\end{equation}
By the definition of $\phi_{\delta\epsilon}$, it is trivial to note
that $V_{\delta\epsilon}\in C^2(\mathbb{R}^n;\mathbb{R}_+)$. For any
$x\in\mathbb{R}^n$ set
\begin{equation*}
(V_{\delta\epsilon})_x(x):=\Big(\frac{\partial
V_{\delta\epsilon}(x)}{\partial x_1},\cdots,\frac{\partial
V_{\delta\epsilon}(x)}{\partial x_n}\Big) \mbox{ and }
(V_{\delta\epsilon})_{xx}(x):=\Big(\frac{\partial^2
V_{\delta\epsilon}(x)}{\partial x_i\partial x_j}\Big)_{n\times n}.
\end{equation*}
We then have
\begin{equation*}
\frac{\partial V_{\delta\epsilon}(x)}{\partial
x_i}=\phi'_{\delta\epsilon}(|x|)\frac{x_i}{|x|}  \mbox{ and
}\frac{\partial^2 V_{\delta\epsilon}(x)}{\partial x_i\partial
x_j}=\phi'_{\delta\epsilon}(|x|)(\delta_{ij}|x|^2-x_ix_j)|x|^{-3}+\phi''_{\delta\epsilon}(|x|)x_ix_j|x|^{-2},
\end{equation*}
for $x\in\mathbb{R}^n$,  $i=1,2,\cdots,n$, where $\delta_{ij}=1$
if $i=j$ or otherwise $0$, and 
\begin{equation}\label{eq6}
0\leq|(V_{\delta\epsilon})_x(x)|\leq1 \mbox{ and }
\|(V_{\delta\epsilon})_{xx}(x)\|\leq
2n\Big(1+\frac{1}{\ln\delta}\Big)\frac{1}{|x|}{\bf1}_{[\epsilon/\delta,\epsilon]}(|x|),\
x\in\mathbb{R}^n.
\end{equation}
For any $t\in[0,T]$, let
\begin{equation*}
Z(t):=X(t)-Y(t),\ \ \ \bar{Z}(t):=Y(t)-\bar{Y}(t)\ \mbox{ and }\
\tilde{Z}(t):=(X(t),\bar{Y}(t))\in\mathbb{R}^{2n}.
\end{equation*}
Application of  the It\^o formula yields that
\begin{equation*}
\begin{split}
V_{\delta\epsilon}(Z(t))
&=\int_0^t\langle(V_{\delta\epsilon})_x(Z(s)),b(X(s),X(s-\tau))-b(\bar{Y}(s),\bar{Y}(s-\tau))\rangle\d s\\
&\quad+\frac{1}{2}\int_0^t\mbox{trace}\{(\sigma(X(s),X(s-\tau))-\sigma(\bar{Y}(s),\bar{Y}(s-\tau)))^*(V_{\delta\epsilon})_{xx}(Z(s))\\
&\quad\times(\sigma(X(s),X(s-\tau))-\sigma(\bar{Y}(s),\bar{Y}(s-\tau)))\}\d
s\\
&\quad+\int_0^t\langle(V_{\delta\epsilon})_x(Z(s)),(\sigma(X(s),X(s-\tau))-\sigma(\bar{Y}(s),\bar{Y}(s-\tau)))\d
W(s)\rangle\\
&:=I_1(t)+I_2(t)+I_3(t).
\end{split}
\end{equation*}
 By \eqref{eq6}, $(A1)$ and the H\"older inequality,   we derive
 that
\begin{equation}\label{eq7}
\begin{split}
\mathbb{E}\Big(\sup\limits_{0\leq s\leq
t}|I_1(s)|^p\Big)&\leq\int_0^t\mathbb{E}|b(X(s),X(s-\tau))-b(\bar{Y}(s),\bar{Y}(s-\tau))|^p\d
s\\
&\leq
C\int_0^t\Big\{\mathbb{E}|Z(s)|^p+\Big(\mathbb{E}V_1^{2p}(\tilde{Z}(s-\tau))\Big)^{\frac{1}{2}}\Big(\mathbb{E}|Z(s-\tau)|^{2p}\Big)^{\frac{1}{2}}\\
&\quad+\mathbb{E}|\bar{Z}(s)|^p+\Big(\mathbb{E}V_1^{2p}(\tilde{Z}(s))\Big)^{\frac{1}{2}}\Big(\mathbb{E}|\bar{Z}(s)|^{2p}\Big)^{\frac{1}{2}}\Big\}\d
s,\ \ \ t\in[0,T],
\end{split}
\end{equation}
and due to $(A2)$ and \eqref{eq6} again that
\begin{equation}\label{eq8}
\begin{split}
\mathbb{E}\Big(\sup\limits_{0\leq s\leq t}|I_2(s)|^p\Big)&\leq
C\int_0^t\mathbb{E}\{\|(V_{\delta\epsilon})_{xx}(Z(s))\|\|\sigma(X(s),X(s-\tau))-\sigma(\bar{Y}(s),\bar{Y}(s-\tau))\|^2\}^p\d
s\\
&\leq
C\mathbb{E}\int_0^t\frac{1}{|Z(s)|^p}\{|Z(s)|^{2p}+V_2^{2p}(\tilde{Z}(s-\tau))|Z(s-\tau)|^{2p}\\
&\quad+|\bar{Z}(s)|^{2p}+V_2^{2p}(\tilde{Z}(s-\tau))|\bar{Z}(s-\tau)|^{2p}\}{\bf1}_{[\epsilon/\delta,\epsilon]}(|Z(s)|)\d
s\\
&\leq
C\int_0^t\Big\{\mathbb{E}|Z(s)|^p+\frac{1}{\epsilon^p}\Big(\mathbb{E}V_2^{4p}(\tilde{Z}(s-\tau))\Big)^{\frac{1}{2}}\Big(\mathbb{E}|Z(s-\tau)|^{4p}\Big)^{\frac{1}{2}}\\
&\quad+\frac{1}{\epsilon^p}\mathbb{E}|\bar{Z}(s)|^{2p}
+\frac{1}{\epsilon^p}\Big(\mathbb{E}V_2^{4p}(\tilde{Z}(s))\Big)^{\frac{1}{2}}\Big(\mathbb{E}|\bar{Z}(s)|^{4p}\Big)^{\frac{1}{2}}\Big\}\d
s, \ \ \ t\in[0,T].
\end{split}
\end{equation}
By virtue of the Burkhold-Davis-Gundy inequality, the H\"older
inequality and \eqref{eq6}, for any $p\geq2$ and $t\in[0,T]$
\begin{equation}\label{eq10}
\begin{split}
\mathbb{E}\Big(\sup\limits_{0\leq s\leq t}|I_3(s)|^p\Big)&\leq
C\mathbb{E}\Big(\int_0^t\|\sigma(X(s),X(s-\tau))-\sigma(\bar{Y}(s),\bar{Y}(s-\tau))\|^2\d
s\Big)^{\frac{p}{2}}\\
&\leq
C\mathbb{E}\int_0^t\|\sigma(X(s),X(s-\tau))-\sigma(\bar{Y}(s),\bar{Y}(s-\tau))\|^{p}\d
s\\
 &\leq C\int_0^t\Big\{\mathbb{E}|Z(s)|^{p}\d s+\Big(\mathbb{E}V_2^{2p}(\tilde{Z}(s-\tau))\Big)^{\frac{1}{2}}\Big(\mathbb{E}|Z(s-\tau)|^{2p}\Big)^{\frac{1}{2}}\\
&\quad+\mathbb{E}|\bar{Z}(s)|^{p}+\Big(\mathbb{E}V_2^{2p}(\tilde{Z}(s))\Big)^{\frac{1}{2}}\Big(\mathbb{E}|\bar{Z}(s)|^{2p}\Big)^{\frac{1}{2}}\Big\}\d
s.
\end{split}
\end{equation}
Furthermore, observe from \eqref{eq19} and \eqref{eq3} that
\begin{equation*}
\mathbb{E}V_1^{2p}(\tilde{Z}(s-\tau))+\mathbb{E}V_2^{4p}(\tilde{Z}(s-\tau))\leq
C
\end{equation*}
and by \eqref{eq9} that $\mathbb{E}|\bar{Z}(t)|^p\leq
C\triangle^{\frac{p}{2}}. $ Then, combining \eqref{eq7}, \eqref{eq8}
with \eqref{eq10}, we thus obtain from \eqref{eq4} that, for any
$t\in[0,T]$ and any $p\geq2$,
\begin{equation*}
\begin{split}
\mathbb{E}\Big(\sup\limits_{0\leq s\leq t}|Z(s)|^p\Big)&\leq
2^{p-1}\Big\{\epsilon^p+\mathbb{E}\Big(\sup\limits_{0\leq s\leq
t}V^p_{\delta\epsilon}(Z(s))\Big)\Big\}\\
&\leq
C\Big\{\epsilon^p+\triangle^{\frac{p}{2}}+\frac{1}{\epsilon^p}\triangle^p+\frac{1}{\epsilon^p}\triangle^{p}\\
&\quad+\int_0^t\mathbb{E}|Z(s)|^p\d
s+\int_0^t\Big(\mathbb{E}|Z(s-\tau)|^{2p}\Big)^{\frac{1}{2}}\d
s+\frac{1}{\epsilon^p}\int_0^t\Big(\mathbb{E}|Z(s-\tau)|^{4p}\Big)^{\frac{1}{2}}\d
s\Big\}.
\end{split}
\end{equation*}
This, together with the Gronwall inequality,   implies  
\begin{equation}\label{eq11}
\begin{split}
\mathbb{E}\Big(\sup\limits_{0\leq s\leq t}|Z(s)|^p\Big)&\leq
C\Big\{\epsilon^p+\triangle^{\frac{p}{2}}+\frac{1}{\epsilon^p}\triangle^p+\int_0^t\Big(\mathbb{E}|Z(s-\tau)|^{2p}\Big)^{\frac{1}{2}}\d s\\
&\quad+\frac{1}{\epsilon^p}\int_0^t\Big(\mathbb{E}|Z(s-\tau)|^{4p}\Big)^{\frac{1}{2}}\d
s\Big\}.
\end{split}
\end{equation}
For any $p\geq2$, let
\begin{equation*}
p_i:=([T/\tau]+2-i)p4^{[T/\tau]+1-i}, \ \ \ i=1,2,\cdots,[T/\tau]+1.
\end{equation*}
 It is easy to see that
\begin{equation}\label{eq12}
4p_{i+1}<p_i \ \mbox{ and } \ p_{[T/\tau]+1}=p, \ \ \
i=1,2,\cdots,[T/\tau].
\end{equation}
Noting that $Z(s-\tau)=0$ for $s\in[0,\tau]$ and taking
$\epsilon=\triangle^{\frac{1}{2}}$ in  \eqref{eq11}, we obtain  that
\begin{equation*}
\mathbb{E}\Big(\sup\limits_{0\leq s\leq \tau}|Z(s)|^{p_1}\Big)\leq
C\triangle^{\frac{p1}{2}}.
\end{equation*}
This, together with  \eqref{eq12} and the H\"older inequality,
further gives that
\begin{equation*}
\begin{split}
\mathbb{E}\Big(\sup\limits_{0\leq s\leq 2\tau}|Z(s)|^{p_2}\Big)&\leq
C\Big\{\triangle^{\frac{p_2}{2}}+\int_0^{2\tau}\Big(\mathbb{E}|Z(s-\tau)|^{p_1}\Big)^{\frac{p_2}{p_1}}\d s\\
&\quad+\triangle^{-\frac{p_2}{2}}\int_0^{2\tau}\Big(\mathbb{E}|Z(s-\tau)|^{p_1}\Big)^{\frac{2p_2}{p_1}}\d
s\Big\}\\
&\leq C\triangle^{\frac{p2}{2}}
\end{split}
\end{equation*}
by taking $\epsilon=\triangle^{\frac{1}{2}}$ in \eqref{eq11}. The
desired assertion then follows by repeating the previous procedures.
\end{proof}

\begin{rem}
{\rm The strong convergence of EM scheme for SDDEs is generally
investigated under local Lipschitz condition and bounded moments of
analytic solutions and numerical solutions, or local Lipschitz
condition and linear growth condition, e.g., \cite{ms03}. In
this section, for a class of SDDEs, which may be {\it highly
nonlinear} with respect to the delay variables, we show the strong
convergence of EM scheme under rather general conditions.  To the best of our knowledge, there are relatively  few results in the existing literature.}
\end{rem}

\begin{rem}
{\rm There are only limited results on convergence order of EM
scheme for SDEs or SDDEs under weaker condition than global
Lipschitz and linear growth condition,  For example, under a
H\"older continuous condition, \cite{gr11} reveals the convergence
order of EM scheme for a class of SDEs, and, with local Lipschitz
constants satisfying a logarithm growth condition, \cite{ym08} and
\cite{bbmy11,jwy09} discuss the convergence rate of EM approximate
methods for SDEs and stochastic functional differential equations
with jumps respectively, where
 {\it linear growth condition} is imposed in \cite{bbmy11,jwy09}.
 While, in this section, under very general conditions we reveal the convergence order of EM scheme for a class of
SDDEs although which are {\it highly nonlinear} with respect to
delay arguments. }
\end{rem}

\section{Convergence Rate for  Pure Jump Case}
In the last section we discuss the strong convergence of EM scheme
for a class of SDDEs, and reveal the convergence rate is
$\frac{1}{2}$ although both the drift coefficient and the diffusion
coefficient may be highly nonlinear with respect to the delay
variables. In this section we turn to the counterpart for  SDDEs with
jumps. We further need to introduce some notation. Let  $\mathcal
{B}(\mathbb{R})$ be the Borel $\sigma$-algebra on $\mathbb{R}$, and
$\lambda(dx)$ a $\sigma$-finite measure defined on $\mathcal
{B}(\mathbb{R})$. Let $p=(p(t)),t\in D_p$, be a stationary $\mathcal
{F}_t$-Poisson point process on $\mathbb{R}$ with characteristic
measure $\lambda(\cdot)$. Denote by $N(dt,du)$ the Poisson counting
measure associated with $p$, i.e., $N(t,U)=\sum_{s\in D_p, s\leq
t}I_{U}(p(s))$ for $U\in\mathcal {B}(\mathbb{R})$. Let
$\tilde{N}(dt,du):=N(dt,du)-dt\lambda(du)$ be the compensated
Poisson measure associated with $N(dt,du)$. In what follows, we
further assume that $\int_U|u|^p\lambda(\d u)<\infty$ for any
$p\geq2$.

In this section we consider SDDE with jumps on $\mathbb{R}^n$
\begin{equation}\label{eq24}
\d X(t)=b(X(t),X(t-\tau))\d t+\int_Uh(X(t),X(t-\tau),u)\tilde{N}(\d
t,\d u), \ \ \ t\in[0,T]
\end{equation}
with initial data $X(\theta)=\xi(\theta),\theta\in[-\tau,0]$, where
$\xi\in\mathscr{C}$. We assume that
\begin{enumerate}
\item[\textmd{(A3)}]
$b:\mathbb{R}^n\times\mathbb{R}^n\rightarrow\mathbb{R}^n$ satisfies the assumption
$(A1)$;
\item[\textmd{(A4)}] $h:\mathbb{R}^n\times\mathbb{R}^n\times
U\rightarrow\mathbb{R}^n$ and there exists $L_3>0$ such that
\begin{equation*}
|h(x_1,y_1,u)-h(x_2,y_2,u)|\leq
(L_3|x_1-x_2|+V_3(y_1,y_2)|y_1-y_2|)|u|
\end{equation*}
for $x_i,y_i\in\mathbb{R}^n,i=1,2$, and $u\in U$, where
$V_3:\mathbb{R}^n\times\mathbb{R}^n\rightarrow\mathbb{R}_+$ such
that
\begin{equation}\label{eq33}
V_3(x,y)\leq K_3(1+|x|^{q_3}+|y|^{q_3})
\end{equation}
for some $K_3>0, q_3\geq1$ and arbitrary $x,y\in\mathbb{R}^n$.
\end{enumerate}

\begin{rem}
{\rm The jump coefficient may be also {\it highly nonlinear} with
respect to the delay arguments, e.g., for $x,y\in\mathbb{R},u\in U$
and $q>1$, $h(x,y,u)=y^qu$ satisfies $(A4)$. }
\end{rem}

Fix $T>0$ and let the stepsize $\triangle$ be defined by
\eqref{eq27}. The EM scheme associated with Eq. \eqref{eq24} is
defined as follows:
\begin{equation}\label{eq22}
\d Y(t)=b(\tilde{Y}(t),\tilde{Y}(t-\tau))\d
t+\int_Uh(\tilde{Y}(t),\tilde{Y}(t-\tau),u)\tilde{N}(\d t,\d u),
\end{equation}
where $\bar{Y}(t):=Y(k\triangle)$ for $
t\in[k\triangle,(k+1)\triangle), k=0,1,\cdots,M-1$, and
$\bar{Y}(\theta)=\xi(\theta),\theta\in[-\tau,0]$.

To reveal  the convergence order of EM scheme \eqref{eq22}, we need two 
auxiliary lemmas,  where the first one is
Bichteler-Jacod inequality for Poisson integrals, e.g., \cite[Lemma
3.1]{mpr10}.

\begin{lem}\label{Kunita inequality}
{\rm Let $\Phi:\mathbb{R}_+\times U\rightarrow\mathbb{R}^n$ and
assume that
\begin{equation*}
\int_0^t\int_U\mathbb{E}|\Phi(s,u)|^p\lambda(\d u)\d s<\infty, \ \ \
t\geq0,\ \ p\geq2.
\end{equation*}
Then there exists $D(p)>0$ such that
\begin{equation*}
\begin{split}
\mathbb{E}\Big(\sup\limits_{0\leq s\leq
t}\Big|\int_0^s\int_U\Phi(r,u)\tilde{N}(\d u,\d s)\Big|^p\Big)&\leq
D(p)\Big\{\mathbb{E}\Big(\int_0^t\int_U|\Phi(s,u)|^2\lambda(\d u)\d
s\Big)^{\frac{p}{2}}\\
&\quad+\mathbb{E}\int_0^t\int_U|\Phi(s,u)|^p\lambda(\d u)\d s\Big\}.
\end{split}
\end{equation*}
}
\end{lem}
Using the Lemma above and the similar argument of Lemma \ref{lemma1.1}, we have
\begin{lem}
{\rm Let $(A3)$ and $(A4)$ hold. Then Eq.\eqref{eq24} has a unique
global solution $(X(t))_{t\in[0,T]}$. Moreover, for any $p\geq2$
there exists $C>0$ such that
\begin{equation}\label{eq25}
\mathbb{E}\Big(\sup_{0\leq t\leq
T}|X(t)|^p\Big)\vee\mathbb{E}\Big(\sup_{0\leq t\leq
T}|Y(t)|^p\Big)\leq C,
\end{equation}
and
\begin{equation}\label{eq26}
\mathbb{E}|Y(t)-\bar{Y}(t)|^p\leq C\triangle.
\end{equation}}
\end{lem}

\begin{rem}
{\rm We remark that  for $p\geq2$  all $p$th-moments of
$Y(t)-\bar{Y}(t)$ are bounded by $\triangle$ up to a constant, which
is completely different from the Brownian motion case \eqref{eq9}.
This is due to the fact that all moments of  the increment
$\tilde{N}((0,(i+1)\triangle],\d u)-\tilde{N}((0,i\triangle],\d u)$
have  order $O(\triangle)$ for $\triangle\in(0,1)$.}
\end{rem}
We now state our main result in this section.
\begin{thm}\label{Th3}
{\rm Let $(A3)$ and $(A4)$ hold. For any $p\geq2$ and arbitrary
$\theta,\alpha\in(0,1)$, there exists $C>0$, independent of
$\triangle$, such that
\begin{equation*}
\mathbb{E}\Big(\sup\limits_{0\leq s\leq T}|X(s)-Y(s)|^p\Big)\leq
C\triangle^{\frac{1}{(1+\theta)^{[T/\tau](1+\alpha)}}}.
\end{equation*}}
\end{thm}
\begin{proof}
The proof of Theorem \ref{Th3} is similar to that of Theorem
\ref{Th2}, while we give a sketch of the proof  to highlight the
differences between the Brwonian motion case. Set
\begin{equation*}
Z(t):=X(t)-Y(t),\ \ \ \bar{Z}(t):=Y(t)-\bar{Y}(t),\ \ \
\tilde{Z}(t):=(X(t),\bar{Y}(t))\in\mathbb{R}^{2n},\ \ \ t\in[0,T].
\end{equation*}
Define  for $t\in[0,T]$
\begin{equation*}
\Gamma_1(t):=b(X(t),X(t-\tau))-b(\bar{Y}(t),\bar{Y}(t-\tau))
\end{equation*}
and
\begin{equation*}
\Gamma_2(t,u):=h(X(t),X(t-\tau),u)-h(\bar{Y}(t),\bar{Y}(t-\tau),u).
\end{equation*}
For $V_{\delta\epsilon}\in C^2(\mathbb{R}^n;\mathbb{R}_+)$, defined
by \eqref{eq21}, the It\^o formula and the Taylor expansion give
that for $t\in[0,T]$
\begin{equation}\label{eq28}
\begin{split}
&V_{\delta\epsilon}(Z(t))
=\int_0^t\langle(V_{\delta\epsilon})_x(Z(s)),\Gamma_1(s)\rangle\d
s+\int_0^t\int_U\{V_{\delta\epsilon}(Z(s)+\Gamma_2(s,u))\\
&\quad-V_{\delta\epsilon}(Z(s))-\langle(V_{\delta\epsilon})_x(Z(s)),\Gamma_2(s,u)\rangle\}\lambda(\d
u)\d s\\
&\quad+\int_0^t\int_U\{V_{\delta\epsilon}(Z(s)+\Gamma_2(s,u))-V_{\delta\epsilon}(Z(s))\}\tilde{N}(\d u,\d s)\\
&=\int_0^t\langle(V_{\delta\epsilon})_x(Z(s)),\Gamma_1(s)\rangle\d
s\\
&\quad+\int_0^t\int_U\Big\{\int_0^1\langle(V_{\delta\epsilon})_x(\theta\Gamma_2(s,u)+Z(s))-(V_{\delta\epsilon})_x(Z(s)),\Gamma_2(s,u)\rangle\d
\theta\Big\}\lambda(\d
u)\d s\\
&\quad+\int_0^t\int_U\Big\{\int_0^1\langle(V_{\delta\epsilon})_x(\theta\Gamma_2(s,u)+Z(s)),\Gamma_2(s,u)\rangle\d
\theta\Big\}\tilde{N}(\d u,\d s).
\end{split}
\end{equation}
By \eqref{eq28}, together with \eqref{eq4} and \eqref{eq6}, we
then deduce that
\begin{equation*}
\begin{split}
|Z(t)|&\leq
\epsilon+V_{\delta\epsilon}(Z(t))\\
&\leq\epsilon+\int_0^t|\Gamma_1(s)|\d
s+2\int_0^t\int_U|\Gamma_2(s,u)|\lambda(\d
u)\d s\\
&\quad+\int_0^t\int_U\Big\{\int_0^1\langle(V_{\delta\epsilon})_x(\theta\Gamma_2(s,u)+Z(s)),\Gamma_2(s,u)\rangle\d
\theta\Big\}\tilde{N}(\d u,\d s),\ \ \ t\in[0,T].
\end{split}
\end{equation*}
Furthermore, note from \eqref{eq19}, \eqref{eq33} and \eqref{eq25}
that for any $q\geq2$
\begin{equation*}
\mathbb{E}\Big(\sup\limits_{0\leq t\leq
T}V_1^q(\tilde{Z}(s-\tau))\Big)+\mathbb{E}\Big(\sup\limits_{0\leq
t\leq T}V_3^q(\tilde{Z}(s-\tau))\Big)\leq C.
\end{equation*}
Consequently, for any $p\geq2$ and $t\in[0,T]$, using \eqref{eq6}
and \eqref{eq26}, Lemma \ref{Kunita inequality} and the H\"older
inequality,   $(A3)$ and $(A4),$  we derive at
\begin{equation*}
\begin{split}
\mathbb{E}\Big(\sup\limits_{0\leq s\leq t}|Z(s)|^p\Big)&\leq2^{p-1}
(\epsilon^p+\mathbb{E}\Big(\sup\limits_{0\leq s\leq t}V_{\delta\epsilon}^p(Z(s))\Big)\\
&\leq C\Big\{\epsilon^p+\int_0^t\mathbb{E}|\Gamma_1(s)|^p\d
s+\int_0^t\int_U\mathbb{E}|\Gamma_2(s,u)|^p\lambda(\d
u)\d s\\
&\quad+\mathbb{E}\Big(\int_0^t\int_U|\Gamma_2(s,u)|^2\lambda(\d u)\d
s\Big)^{\frac{p}{2}}\Big\}\\
&\leq C\Big\{\epsilon^p+\int_0^t\mathbb{E}|\Gamma_1(s)|^p\d
s+\int_0^t\int_U\mathbb{E}|\Gamma_2(s,u)|^p\lambda(\d u)\d s\Big\}\\
&\leq
C\Big\{\epsilon^p+\int_0^t\mathbb{E}(|X(s)-\bar{Y}(s)|\\
&\quad+V_1(\tilde{Z}(s-\tau))|X(s-\tau)-\tilde{Y}(s-\tau)|)^p\d
s+\int_0^t\mathbb{E}(|X(s)-\bar{Y}(s)|\\
&\quad+V_3(\tilde{Z}(s-\tau))|X(s-\tau)-\tilde{Y}(s-\tau)|)^p\d
s\Big\}\\
&\leq
C\Big\{\epsilon^p+\triangle+\int_0^t\{\mathbb{E}|Z(s)|^p+\mathbb{E}(V_1^p(\tilde{Z}(s-\tau))|Z(s-\tau)|^p)\\
&\quad+\mathbb{E}(V_1^p(\tilde{Z}(s-\tau))|\bar{Z}(s-\tau)|^p)+\mathbb{E}(V_3^p(\tilde{Z}(s-\tau))|Z(s-\tau)|^p)\\
&\quad+\mathbb{E}(V_3^p(\tilde{Z}(s-\tau))|\bar{Z}(s-\tau)|^p)\}\d
s\Big\}\\
&\leq
C\bigg\{\epsilon^p+\triangle+\int_0^t\mathbb{E}|Z(s)|^p\d s\\
&\quad+\int_0^t\Big\{\Big(\mathbb{E}|Z(s-\tau)|^{p(1+\theta)}\Big)^{\frac{1}{1+\theta}}+\Big(\mathbb{E}|\bar{Z}(s-\tau)|^{p(1+\theta)}
\Big)^{\frac{1}{1+\theta}}\Big\}\d s\bigg\},
\end{split}
\end{equation*}
where $\theta\in(0,1)$ is an arbitrary constant. An application of
the Gronwall inequality then gives that
\begin{equation}\label{eq31}
\begin{split}
\mathbb{E}\Big(\sup\limits_{0\leq s\leq t}|Z(s)|^p\Big)&\leq
C\bigg\{\triangle+\int_0^t\Big\{\Big(\mathbb{E}|Z(s-\tau)|^{p(1+\theta)}\Big)^{\frac{1}{1+\theta}}\\
&\quad+\Big(\mathbb{E}|\bar{Z}(s-\tau)|^{p(1+\theta)}
\Big)^{\frac{1}{1+\theta}}\Big\}\d s\bigg\},\ \ \ t\in[0,T]
\end{split}
\end{equation}
by taking $\epsilon=\triangle^{\frac{1}{p}}$. For $\theta\in(0,1)$
in \eqref{eq31} and any $\alpha\in(0,1)$, let
\begin{equation*}
p_i:=p(1+\theta)^{([T/\tau]+1-i)(1+\alpha)}, \ \ \
i=1,2,\cdots,[T/\tau]+1.
\end{equation*}
It is trivial to see that
\begin{equation}\label{eq32}
(1+\theta)p_{i+1}<p_i \ \mbox{ and } \ p_{[T/\tau]+1}=p, \ \ \
i=1,2,\cdots,[T/\tau].
\end{equation}
Noting that $Z(t)=\bar{Z}(t)=0$ for $t\in[-\tau,0]$, by \eqref{eq31}
we clearly get
\begin{equation*}
\mathbb{E}\Big(\sup\limits_{0\leq s\leq \tau}|Z(s)|^{p_1}\Big)\leq
C\triangle.
\end{equation*}
This, together with \eqref{eq26}, \eqref{eq31} and the H\"older
inequality, yields that
\begin{equation}\label{eq34}
\begin{split}
\mathbb{E}\Big(\sup\limits_{0\leq s\leq 2\tau}|Z(s)|^{p_2}\Big)&\leq
C\Big\{\triangle+\triangle^{\frac{1}{1+\theta}}+\int_0^{2\tau}\Big(\mathbb{E}|Z(s-\tau)|^{p_2(1+\theta)}\Big)^{\frac{1}{1+\theta}}\d s\Big\}\\
&\leq C\Big\{\triangle+\triangle^{\frac{1}{1+\theta}}+\int_0^{2\tau}\Big(\mathbb{E}|Z(s-\tau)|^{p_1}\Big)^{\frac{p_2}{p_1}}\d s\Big\}\\
&\leq C\{\triangle+\triangle^{\frac{1}{1+\theta}}+\triangle^{\frac{p_2}{p_1}}\}\\
&\leq C\triangle^{\frac{p_2}{p_1}},
\end{split}
\end{equation}
where the last step is due to \eqref{eq32}. Similarly, we have from
\eqref{eq31}-\eqref{eq34} that
\begin{equation*}
\begin{split}
\mathbb{E}\Big(\sup\limits_{0\leq s\leq 3\tau}|Z(s)|^{p_3}\Big)&\leq
C\Big\{\triangle+\triangle^{\frac{1}{1+\theta}}+\int_0^{3\tau}\Big(\mathbb{E}|Z(s-\tau)|^{p_3(1+\theta)}\Big)^{\frac{1}{1+\theta}}\d s\Big\}\\
&\leq C\Big\{\triangle+\triangle^{\frac{1}{1+\theta}}+\int_0^{3\tau}\Big(\mathbb{E}|Z(s-\tau)|^{p_2}\Big)^{\frac{p_3}{p_2}}\d s\Big\}\\
&\leq C\{\triangle+\triangle^{\frac{1}{1+\theta}}+\triangle^{\frac{p_3}{p_1}}\}\\
&\leq C\triangle^{\frac{p_3}{p_1}}.
\end{split}
\end{equation*}
Following the previous procedures gives that
\begin{equation*}
\mathbb{E}\Big(\sup\limits_{0\leq s\leq T}|Z(s)|^p\Big)\leq
C\triangle^{\frac{1}{(1+\theta)^{[T/\tau](1+\alpha)}}},
\end{equation*}
and the proof is therefore complete.
\end{proof}

\begin{rem}
{\rm By Theorem \ref{Th3}, with $p\geq2$ increasing the convergence
rate of EM scheme \eqref{eq22} is decreasing, which is quite
different from the Brownian motion case with a constant order
$\frac{1}{2}$, and it is therefore best to use the mean-square
convergence for the jump case. On the other hand, we reveal that the
order of mean-square convergence is close to $\frac{1}{2}$ although
the jump diffusion may be highly nonlinear with respect to the delay
variables. }
\end{rem}

\end{document}